\documentclass{amsart}%
\usepackage{amsmath}
\usepackage{amssymb}
\usepackage{amsfonts}
\usepackage{graphicx}%
\providecommand{\U}[1]{\protect\rule{.1in}{.1in}}
\newtheorem{theorem}{Theorem}
\theoremstyle{plain}

\newtheorem{corollary}{Corollary}

\newtheorem{lemma}{Lemma}

\newtheorem{proposition}{Proposition}
\newtheorem{remark}{Remark}

\numberwithin{equation}{section}
\begin{document}
\title[Irrationality exponents of semi-regular continued fractions]{ Irrationality exponents of semi-regular continued fractions}
\author{Daniel Duverney}
\address{110, rue du chevalier fran\c{c}ais, 59000 Lille, France}
\email{daniel.duverney@orange.fr}
\author{Iekata Shiokawa}
\address{13-43, Fujizuka-cho, Hodogaya-ku, Yokohama 240-0031, Japan}
\email{shiokawa@beige.ocn.ne.jp}
\date{January 30, 2022}

\begin{abstract}
We prove that the formula giving the exact value of the irrationality exponent
of regular continued fractions remains valid for semi-regular continued
fractions satisfying certain conditions.

\end{abstract}
\maketitle

\section{Introduction}

An infinite continued fraction%
\begin{equation}
\alpha:=b_{0}+\frac{a_{1}}{b_{1}}%
\genfrac{}{}{0pt}{}{{}}{+}%
\frac{a_{2}}{b_{2}}%
\genfrac{}{}{0pt}{}{{}}{+\cdots+}%
\frac{a_{n}}{b_{n}}%
\genfrac{}{}{0pt}{}{{}}{+\cdots}
\label{Int1}%
\end{equation}
is called \textit{semi-regular} if $b_{0}\in\mathbb{Z}$ and%
\begin{equation}
a_{n}\in\left\{  -1,1\right\}  ,\qquad b_{n}\in\mathbb{Z}_{>0},\qquad
b_{n}+a_{n+1}\geq1\qquad\left(  n\geq1\right)  , \label{Int2}%
\end{equation}
with the additional condition%
\begin{equation}
b_{n}+a_{n+1}\geq2\qquad\text{infinitely often.} \label{Cond1}%
\end{equation}
It is known that semi-regular continued fractions (SRCF) are well defined and
convergent by Tietze Theorem \cite{Tietze}. Examples of SRCF (among others) are:

(i) The \textit{regular }continued fractions (RCF), where $a_{n}=1$ for every
$n\geq1.$

(ii) The \textit{negative }continued fractions (NCF), also known as
\textit{backward }continued fractions, where $a_{n}=-1$ for every $n\geq1$
\cite{Cahen}$.$

(iii) The \textit{nearest integer }continued fractions (NICF), where
$b_{n}+a_{n+1}\geq2$ and $b_{n}\geq2$ for every $n\geq1$ \cite{Per}.

(iv) The \textit{singular} continued fractions (SCF), where $b_{n}+a_{n}\geq2$
and $b_{n}\geq2$ for every $n\geq1$ \cite{Per}.

(v) The \textit{Lehner} continued fractions (LCF), where $(b_{n}%
,a_{n+1})=(1,1)$ or $(2,-1)$ for every $n\geq0$ (\cite{Daj}, \cite{Leh}%
).\medskip

It should be noted that the condition (\ref{Cond1}) is always realized if
there exist infinitely many $n$ such that $a_{n}=1.$ So (\ref{Cond1}) is
always realized, except if there exists $n$ such that
\begin{equation}
x_{n}:=\frac{a_{n+1}}{b_{n+1}}%
\genfrac{}{}{0pt}{}{{}}{+}%
\frac{a_{n+2}}{b_{n+2}}%
\genfrac{}{}{0pt}{}{{}}{+}%
\frac{a_{n+3}}{b_{n+3}}%
\genfrac{}{}{0pt}{}{{}}{+\cdots}%
\qquad\left(  n\geq0\right)  \label{xn2}%
\end{equation}
is a NCF and $b_{k}=2$ for all large $k$ (see Remark \ref{RemIrrat} below).
Hence in the definition of SRCF the condition (\ref{Cond1}) can be replaced
by\medskip

\qquad\textit{If }$a_{n}=-1$\textit{\ for all large }$n,$\textit{\ then
}$b_{n}\geq3$\textit{\ infinitely often.\medskip}

This shows that NCF are of special importance among SRCF. \medskip

The story of SRCF seems to begin with Lagrange, who mentioned in
\cite[Sections 6 and 7]{Lag} the possibility of introducing signs in the
partial numerators of regular continued fractions and gave the way for
transforming such continued fractions into regular ones. Tietze \cite{Tietze}
proved the convergence of SRCF and the irrationality of their values. Chapter
5 of Perron's classical book \cite{Per} is entirely devoted to SRCF. Holzbaur
and Riederle \cite{Holz} characterized SRCF which represent quadratic
irrationals. Pinner \cite{Pin} computed certain inhomogeneous approximation
constants by using NCF expansions. Sarma and Kushwaha \cite{SK} examined
precisely the algorithm to convert a SRCF into regular, even or odd continued fractions.

SRCF have also been widely discussed from the metric theoretical point of view
(see \cite{Daj} and \cite{Kr}). In particular, the map attached to NCF defined
by R\'{e}nyi \cite{Ren} has been studied by many authors (see \cite{Adler} and
the recent papers \cite{Ito}, \cite{Ta}).\medskip

Let $\alpha$ be the SRCF defined by (\ref{Int1}). As usual, we define for
$n\geq1$%
\begin{equation}
\left\{
\begin{array}
[c]{lll}%
p_{-1}=1, & p_{0}=b_{0}, & p_{n}=b_{n}p_{n-1}+a_{n}p_{n-2},\\
q_{-1}=0, & q_{0}=1, & q_{n}=b_{n}q_{n-1}+a_{n}q_{n-2}.
\end{array}
\right.  \label{Int5}%
\end{equation}
It is well known (\cite{Duv}, \cite{JoTh}, \cite{Per}) that%
\begin{equation}
b_{0}+\frac{a_{1}}{b_{1}}%
\genfrac{}{}{0pt}{}{{}}{+}%
\frac{a_{2}}{b_{2}}%
\genfrac{}{}{0pt}{}{{}}{+\cdots+}%
\frac{a_{n}}{b_{n}}=\frac{p_{n}}{q_{n}}\qquad\left(  n\geq1\right)  ,
\label{Int6}%
\end{equation}
and an easy induction using (\ref{Int5}) shows that%
\begin{equation}
p_{n}q_{n-1}-p_{n-1}q_{n}=\left(  -1\right)  ^{n-1}a_{1}a_{2}\cdots
a_{n}\qquad\left(  n\geq1\right)  , \label{Int7}%
\end{equation}
which yields immediately%
\begin{equation}
\frac{p_{n}}{q_{n}}=b_{0}+\sum_{k=1}^{n}\frac{\left(  -1\right)  ^{k-1}%
a_{1}a_{2}\cdots a_{k}}{q_{k-1}q_{k}}. \label{Int8}%
\end{equation}
By (\ref{Int7}), we observe that $p_{n}$ and $q_{n}$ are coprime for all
$n\geq1.\medskip$

For a real number $\alpha,$ the irrationality exponent $\mu\left(
\alpha\right)  $ is defined by the infimum of the set of numbers $\mu$ for
which the inequality%
\begin{equation}
\left\vert \alpha-\frac{p}{q}\right\vert <\frac{1}{q^{\mu}} \label{1.1}%
\end{equation}
has only finitely many rational solutions $p/q$, or equivalently the supremum
of the set of numbers $\mu$ for which the inequality (\ref{1.1}) has
infinitely many solutions. If $\alpha$ is irrational, then $\mu\left(
\alpha\right)  \geq2$. If $\alpha$ is a real algebraic irrational number, then
$\mu\left(  \alpha\right)  =2$ by Roth's theorem \cite{Roth}. If $\mu\left(
\alpha\right)  =\infty,$ then $\alpha$ is called a Liouville number.

When $\alpha$ defined by (\ref{Int1}) is a regular continued fraction, that is
when $a_{n}=1$ for all $n\geq1,$ it is known that%
\begin{equation}
\mu\left(  \alpha\right)  =1+\limsup_{n\rightarrow\infty}\frac{\log q_{n+1}%
}{\log q_{n}}=2+\limsup_{n\rightarrow\infty}\frac{\log b_{n+1}}{\log q_{n}}
\label{Form}%
\end{equation}
(see for example \cite[Theorem 1]{Sondow}). The main purpose of this paper is
to examine to what extent this formula applies to SRCF. Our main result is

\begin{theorem}
\label{Th1}Let $\alpha$ be the SRCF defined by (\ref{Int1}) with (\ref{Int2})
and (\ref{Cond1}). Assume that one of the following conditions holds:

(A) The number of consecutive $n$ such that $a_{n}=-1$ is bounded.

(B) $b_{n}\geq2$ for all large $n$ and the number of consecutive $n$ such that
$b_{n}=2$ is bounded.

(C) $b_{n}+a_{n}\geq2$ for all large $n.$

(D) $b_{n}\geq2$ and $b_{n}+a_{n+1}\geq2$ for all large $n.$

Then (\ref{Form}) holds.
\end{theorem}

By $(C)$ and $(D)$, we see that (\ref{Form}) holds for all SCF and NICF
respectively, but it is not valid in general for LCF. Indeed, (\ref{Form})
implies that $\mu\left(  \alpha\right)  =2$ as soon as $b_{n}$ is bounded
since $q_{n}\rightarrow\infty$ by Tietze Theorem. As any irrational number
$\alpha\in\left]  1,2\right[  $ can be expanded uniquely in a LCF \cite{Leh},
(\ref{Form}) is false for LCF if $\mu\left(  \alpha\right)  >2.$ On the other
hand, $(A)$ implies that $\mu\left(  \alpha\right)  =2$ for all LCF for which
the number of consecutive $n$ such that $\left(  b_{n},a_{n+1}\right)
=\left(  2,-1\right)  $ is bounded.

Formula (\ref{Form}) is not valid either for all NCF (see Section \ref{SecNCF}
below). However, $(B)$ implies that (\ref{Form}) holds for all NCF for which
the number of consecutive $n$ such that $b_{n}=2$ is bounded.

By adding a condition, we can deduce from Theorem \ref{Th1} an expression of
the irrationality exponent of $\alpha$ in terms of the $b_{n}$'s.

\begin{corollary}
\label{Cor1.1}Let $\alpha$ be as in Theorem \ref{Th1}. Assume that
\[
\lim_{n\rightarrow\infty}\frac{n}{\log\left(  b_{1}b_{2}\cdots b_{n}\right)
}=0.
\]
Then%
\begin{equation}
\mu\left(  \alpha\right)  =2+\limsup_{n\rightarrow\infty}\frac{\log b_{n+1}%
}{\log(b_{1}b_{2}\cdots b_{n})}. \label{Form1}%
\end{equation}

\end{corollary}

Note that (\ref{Form1}) also results from \cite[Corollary 4]{DS} when the
series $\sum$ $\left(  b_{n}b_{n+1}\right)  ^{-1}$ is convergent and, in the
case of regular continued fractions, from \cite[Corollary 2]{Sondow}.

In order to prove Theorem \ref{Th1}, we will need to present in some details
the bases of the theory of SRCF. This theory is exposed in Perron's book
\cite{Per} and, to our knowledge, it has not been updated since then. Besides,
some proofs can be simplified, as we are interested here only in infinite
SRCF. Thus, we will study first some basic properties of infinite SRCF in
Section \ref{SecConv}. In Sections \ref{SecNCF} and \ref{SecLCF} respectively
we will prove formulas for the irrationality exponents of NCF and LCF by
transforming them into RCF. In particular, we will give in Theorem
\ref{ThHancl} examples of NCF for which (\ref{Form}) is not valid. In Section
\ref{SecIrrExp}, we will state and prove Theorem \ref{ThIrratGen}, which gives
a rather general formula allowing to compute irrationality exponents. In
Section \ref{SecProof1} we will prove Theorem \ref{Th1} and Corollary
\ref{Cor1.1}. Finally, we will give some examples of application of Corollary
\ref{Cor1.1} in Section \ref{SecEx}.

\section{Convergence and irrationality of infinite SRCF}

\label{SecConv}In this section we recall some basic results on infinite SRCF.

\begin{proposition}
\label{Lem1}Let $q_{n}$ be defined by (\ref{Int5}). We have for $n\geq0$%
\[
\left\{
\begin{array}
[c]{l}%
q_{n}\geq1\\
q_{n}+a_{n+1}q_{n-1}\geq1\\
q_{n+1}>q_{n}\quad\text{if}\quad b_{n+1}\geq2\text{ or }a_{n+1}=1\\
q_{n+2}>q_{n}>q_{n+1}\quad\text{if}\quad b_{n+1}=1\text{ and }a_{n+1}=-1.
\end{array}
\right.
\]

\end{proposition}

\begin{proof}
We prove first by induction that
\begin{equation}
q_{n}\geq1\quad\text{and}\quad q_{n}+a_{n+1}q_{n-1}\geq1 \label{Lemm2}%
\end{equation}
by following Offer \cite{Offer}. When $n=0,$ $q_{0}+a_{1}q_{-1}=q_{0}=1$ by
(\ref{Int5}). Assume that (\ref{Lemm2}) holds for some $n\geq0.$ Then by
(\ref{Int1}), (\ref{Int2}) and (\ref{Int5})%
\[
q_{n+1}=b_{n+1}q_{n}+a_{n+1}q_{n-1}\geq q_{n}+a_{n+1}q_{n-1}\geq1
\]
and $q_{n+1}+a_{n+2}q_{n}=\left(  b_{n+1}+a_{n+2}\right)  q_{n}+a_{n+1}%
q_{n-1}\geq q_{n}+a_{n+1}q_{n-1}\geq1,$ which proves (\ref{Lemm2}). If
$b_{n+1}\geq2,$ then by (\ref{Lemm2})%
\[
q_{n+1}\geq2q_{n}+a_{n+1}q_{n-1}\geq q_{n}+1,
\]
and the same inequation holds evidently if $a_{n+1}=1.$ Finally, if
$b_{n+1}=1$ and $a_{n+1}=-1,$ then $a_{n+2}=1$ by (\ref{Int2}), and so
$q_{n+2}=b_{n+2}q_{n+1}+q_{n}>q_{n}$ by (\ref{Lemm2}), which completes the
proof of Proposition \ref{Lem1}.
\end{proof}

Proposition \ref{Lem1} shows that the the convergents of the infinite SRCF
defined by (\ref{Int1}) and (\ref{Int2}) are well defined since $q_{n}\geq1$
for all $n\geq0.$ It also shows that the sequence $q_{n}$ is not always
increasing, for example in the case of the Lehner continued fraction%
\[
1+\frac{\sqrt{2}}{2}=1+\frac{1}{2}%
\genfrac{}{}{0pt}{}{{}}{-}%
\frac{1}{1}%
\genfrac{}{}{0pt}{}{{}}{+}%
\frac{1}{2}%
\genfrac{}{}{0pt}{}{{}}{-}%
\frac{1}{1}%
\genfrac{}{}{0pt}{}{{}}{+}%
\frac{1}{2}%
\genfrac{}{}{0pt}{}{{}}{-}%
\frac{1}{1}%
\genfrac{}{}{0pt}{}{{}}{+\cdots}%
.
\]
However, the sequence $q_{n}$ is increasing in the cases of regular, negative,
nearest integer and singular SRCF.

\begin{proposition}
\label{ThDef} (Tietze Theorem) The infinite SRCF defined by (\ref{Int1}) and
(\ref{Int2}) is convergent and $\lim_{n\rightarrow\infty}q_{n}=+\infty.$
\end{proposition}

The proof of Tietze Theorem in the general case uses continuants (\cite{Per},
\cite{Offer}). It is rather long and we don't reproduce it here. However, when
the sequence $q_{n}$ is increasing, the convergence of $p_{n}/q_{n}$ follows
immediately from (\ref{Int8}).

\begin{proposition}
\label{Lem-x-n,k}For $n\geq0,$ let $x_{n}$ be defined by (\ref{xn2}). Then for
$n\geq0$%
\[
0<x_{n}\leq1\text{ if }a_{n+1}=1,\qquad-1\leq x_{n}<0\text{ if }a_{n+1}=-1.
\]

\end{proposition}

\begin{proof}
Define for $n\geq1$ and $k\geq0$%
\begin{equation}
x_{n,k}:=\frac{a_{n+1}}{b_{n+1}}%
\genfrac{}{}{0pt}{}{{}}{+}%
\frac{a_{n+2}}{b_{n+2}}%
\genfrac{}{}{0pt}{}{{}}{+\cdots+}%
\frac{a_{n+k}}{b_{n+k}}. \label{x-n-k-0}%
\end{equation}
As $x_{n}=\lim_{k\rightarrow\infty}x_{n,k}$ by Proposition \ref{ThDef}, we
have to prove that for $n\geq0$ and $k\geq1,$%
\begin{equation}
0<x_{n,k}\leq1\text{ if }a_{n+1}=1,\qquad-1\leq x_{n,k}<0\text{ if }%
a_{n+1}=-1. \label{x-n-k}%
\end{equation}
The proof is by induction on $k.$ If $k=1,$ then $x_{n,1}=a_{n+1}/b_{n+1}$ and
so (\ref{Int2}) implies (\ref{x-n-k}) for $k=1$ and any $n\geq0.$ Assume that
(\ref{x-n-k}) holds for some $k\geq1$ and any $n\geq0.$ Then by (\ref{x-n-k-0}%
) we have%
\begin{equation}
x_{n,k+1}=\frac{a_{n+1}}{b_{n+1}+x_{n+1,k}}, \label{x-n-k-1}%
\end{equation}
where for any $n\geq0$%
\begin{equation}
0<x_{n+1,k}\leq1\text{ if }a_{n+2}=1,\qquad-1\leq x_{n+1,k}<0\text{ if
}a_{n+2}=-1. \label{x-n-k-2}%
\end{equation}
If $b_{n+1}\geq2,$ then by (\ref{x-n-k-2})%
\begin{equation}
b_{n+1}+x_{n+1,k}\geq1. \label{x-n-k-3}%
\end{equation}
Otherwise, we have $b_{n+1}=1$ and therefore $a_{n+2}=1$ by (\ref{Int2}), and
so we find again (\ref{x-n-k-3}) by the first inequality in (\ref{x-n-k-2}).
Hence (\ref{x-n-k-1}) and (\ref{x-n-k-3}) imply%
\[
0<x_{n,k+1}\leq1\text{ if }a_{n+1}=1,\qquad-1\leq x_{n,k+1}<0\text{ if
}a_{n+1}=-1,
\]
which proves (\ref{x-n-k}) by induction, and Proposition \ref{Lem-x-n,k} follows.
\end{proof}

\begin{proposition}
\label{ThIrrat}The value of the infinite SRCF $\alpha$ defined by (\ref{Int1})
with (\ref{Int2}) and (\ref{Cond1}) is irrational.
\end{proposition}

\begin{proof}
Let $x_{n}$ be defined in (\ref{xn2}). It is well known, and easy to check by
induction, that for all $n\geq-1$%
\[
\alpha=\frac{p_{n+1}+x_{n+1}p_{n}}{q_{n+1}+x_{n+1}q_{n}}.
\]
This yields immediately%
\begin{equation}
\left\vert \alpha-\frac{p_{n}}{q_{n}}\right\vert =\frac{1}{q_{n}\left\vert
\xi_{n+1}\right\vert }, \label{Ksi2}%
\end{equation}
where $\xi_{n+1}:=q_{n+1}+x_{n+1}q_{n}$ $\left(  n\geq0\right)  .$ Assume
first that there exist infinitely many $n$ such that $a_{n}=1.$ Then by
Proposition \ref{Lem-x-n,k} there are infinitely many $n$ such that
$x_{n+1}>0$ and therefore by (\ref{Ksi2})%
\begin{equation}
0<\left\vert \alpha-\frac{p_{n}}{q_{n}}\right\vert \leq\frac{1}{q_{n}q_{n+1}%
}\leq\frac{2}{q_{n}q_{n+1}} \label{Ksi3}%
\end{equation}
for infinitely many $n.$ Now if $a_{n}=-1$ for all large $n,$ there is no loss
of generality in assuming that $a_{n}=-1$ for all $n\geq1.$ In this case
$x_{n}$ is a negative continued fraction and there exist infinitely many
$n\geq N$ such that $b_{n+1}\geq3,$ in which case%
\[
q_{n+1}=b_{n+1}q_{n}-q_{n-1}\geq2q_{n}+q_{n}-q_{n-1}>2q_{n},
\]
since $q_{n}>q_{n-1}$ by Proposition \ref{Lem1}. As $-1<x_{n+1}\leq0$ by
Proposition \ref{Lem-x-n,k}, we get
\[
\xi_{n+1}=q_{n+1}+x_{n+1}q_{n}>q_{n+1}-q_{n}\geq\frac{1}{2}q_{n+1}%
\]
and therefore (\ref{Ksi3}) holds again for infinitely many $n.$ Now, assume
that $\alpha$ is a rational number, $\alpha=c/d$ with $d>0.$ Then by
(\ref{Ksi3}) there exist infinitely many $n$ such that $1\leq\left\vert
cq_{n}-dp_{n}\right\vert \leq2d/q_{n+1}$ since $cq_{n}-dp_{n}$ is a non-zero
integer. This yields a contradiction for $n$ large because $\lim
_{n\rightarrow\infty}q_{n+1}=+\infty.$
\end{proof}

\begin{remark}
\label{RemIrrat}The continued fraction%
\begin{equation}
\omega:=\frac{1}{2}%
\genfrac{}{}{0pt}{}{{}}{-}%
\frac{1}{2}%
\genfrac{}{}{0pt}{}{{}}{-\cdots-}%
\frac{1}{2}%
\genfrac{}{}{0pt}{}{{}}{-\cdots}
\label{Omega}%
\end{equation}
is convergent by Proposition \ref{ThDef}, and clearly $\omega=1/(2-\omega),$
whence $\omega=1.$ This explains the condition (\ref{Cond1}) in the definition
of SRCF.
\end{remark}

\section{Negative continued fractions}

\label{SecNCF}In this section, we study in more detail negative continued
fractions (NCF), in order to show that formula (\ref{Form}) may not hold in
this case. We will use the following notation:%
\[
b_{0}-\frac{1}{b_{1}}%
\genfrac{}{}{0pt}{}{{}}{-}%
\frac{1}{b_{2}}%
\genfrac{}{}{0pt}{}{{}}{-\cdots-}%
\frac{1}{b_{n}}=\left[  b_{0};b_{1},b_{2},\ldots,b_{n}\right]  ^{-},
\]
which is similar to that for regular continued fractions (RCF):%
\[
b_{0}+\frac{1}{b_{1}}%
\genfrac{}{}{0pt}{}{{}}{+}%
\frac{1}{b_{2}}%
\genfrac{}{}{0pt}{}{{}}{+\cdots+}%
\frac{1}{b_{n}}=\left[  b_{0};b_{1},b_{2},\ldots,b_{n}\right]  .
\]
As already observed, NCF are of special importance among SRCF. They present
two characteristics:\medskip

(a) $b_{n}\geq3$\textit{\ }infinitely often, in order for NCF to be irrational
(Remark \ref{RemIrrat}).\medskip\ 

(b) For all $n\geq0,$ the convergent $p_{n}/q_{n}$ is an approximation by
excess of the infinite NCF $\alpha$ defined in (\ref{Int1}). This results
immediately from (\ref{Int8}).\medskip

A constructive process, which goes back to Lagrange \cite{Lag}, allows to
transform any NCF into a RCF by using the following two elementary formulas,
which are easy to prove. The first one is%
\begin{equation}
x-\frac{1}{y}=x-1+\frac{1}{1}%
\genfrac{}{}{0pt}{}{{}}{+}%
\frac{1}{y-1}. \label{F1}%
\end{equation}
The second one is%
\begin{equation}
x+\frac{1}{0}%
\genfrac{}{}{0pt}{}{{}}{+}%
\frac{1}{y}=x+y. \label{F2}%
\end{equation}
Assume first that $b_{0},b_{1},\ldots,b_{n}$ are indeterminates, and let%
\[
T_{n}:=\left[  b_{0};b_{1},b_{2},\ldots,b_{n}\right]  ^{-}.
\]
By using (\ref{F1}), we get easily%
\begin{equation}
T_{n}=\left[  b_{0}-1;\underline{1,b_{1}-2},\underline{1,b_{2}-2}%
,\ldots,\underline{1,b_{n}-2},1\right]  . \label{Trans1}%
\end{equation}
Now assume that $b_{0}$ is a rational integer and that $b_{1},b_{2}%
,\ldots,b_{n},\ldots$ are rational integers with $b_{n}\geq2$ for all $n\geq1$
and $b_{n}\geq3$ for infinitely many $n\geq1.$ Let $n_{k}$ be defined by
$n_{0}=0$ and for all $k\geq1$%
\begin{equation}
b_{n_{k}}\geq3\quad\text{and}\quad b_{n}=2\text{ if }n_{k-1}<n<n_{k}.
\label{n_k}%
\end{equation}
Then (\ref{Trans1}) yields by using (\ref{F2})%
\begin{equation}
T_{n_{k}}=\left[  b_{0}-1;\underline{n_{1}-n_{0},b_{n_{1}}-2},\underline{n_{2}%
-n_{1},b_{n_{2}}-2},\ldots,\underline{n_{k}-n_{k-1},b_{n_{k}}-2},1\right]  .
\label{Trans2}%
\end{equation}
Consider the infinite RCF%
\begin{align*}
\beta &  =\left[  b_{0}-1;\underline{n_{1}-n_{0},b_{n_{1}}-2},\underline{n_{2}%
-n_{1},b_{n_{2}}-2},\ldots,\underline{n_{k}-n_{k-1},b_{n_{k}}-2},\ldots\right]
\\
&  =\left[  c_{0};c_{1},c_{2},c_{3},c_{4},\ldots,c_{2k-1},c_{2k}%
,\ldots\right]  ,
\end{align*}
where $c_{0}=b_{0}-1$ and for $k\geq1$%
\begin{equation}
c_{2k-1}=n_{k}-n_{k-1},\quad c_{2k}=b_{n_{k}}-2 \label{ck}%
\end{equation}
with $n_{k}$ defined in (\ref{n_k}). We observe by (\ref{Trans2}) that%
\[
T_{n_{k}}=\left[  c_{0};c_{1},c_{2},c_{3},c_{4},\ldots,c_{2k-1},c_{2k}%
,1\right]  .
\]
Let $P_{n}/Q_{n}$ be the \textit{n}-th convergent of $\beta.$ Then%
\[
\left\vert T_{n_{k}}-\frac{P_{2k}}{Q_{2k}}\right\vert =\left\vert \frac
{P_{2k}+P_{2k-1}}{Q_{2k}+Q_{2k-1}}-\frac{P_{2k}}{Q_{2k}}\right\vert =\frac
{1}{Q_{2k}\left(  Q_{2k}+Q_{2k-1}\right)  }.
\]
Hence $\lim_{k\rightarrow\infty}\left(  T_{n_{k}}-P_{2k}/Q_{2k}\right)  =0,$
which proves that%
\begin{equation}
\left[  b_{0};b_{1},b_{2},\ldots,b_{n},\ldots\right]  ^{-}=\left[  c_{0}%
;c_{1},c_{2},\ldots,c_{n},\ldots\right]  , \label{Trans10}%
\end{equation}
where the $c_{n}$ are defined by $c_{0}=b_{0}-1$ and (\ref{ck}). This formula
allows to transform any infinite NCF into an infinite RCF.

\begin{remark}
\label{RemIto}Formula (\ref{Trans10}) is equivalent to Formula (6) in
\cite{Ito} (see also \cite[Proposition 2]{Daj}), which allows to transform any
RCF into a NCF:%
\[
\left[  0;a_{1},a_{2},\ldots\right]  =[1;\underset{\left(  a_{1}-1\right)
\text{ \textit{times}}}{\underline{2,\ldots,2}},a_{2}+2,\underset{\left(
a_{3}-1\right)  \text{ \textit{times}}}{\underline{2,\ldots,2}},a_{4}%
+2,\ldots]^{-}.
\]

\end{remark}

\begin{remark}
\label{Transf}More generally, any SRCF can be transformed into a regular one
by using (\ref{F1}) and (\ref{F2}), as shown in \cite[p. 159]{Per}. The
converse is also true \cite{Daj}.
\end{remark}

Now we can show that (\ref{Form}) doesn't hold in general for NCF. For this,
we consider an arbitrary increasing sequence $n_{k}$ with $n_{0}=0$ and define
$b_{n}$ by%
\begin{equation}
b_{n_{k}}=3,\qquad b_{n}=2\text{\quad if\quad}n_{k-1}<n<n_{k}. \label{bn}%
\end{equation}
Let $\alpha$ be the infinite NCF defined by (\ref{bn}). Assume that
(\ref{Form}) is true for every NCF. Then $\mu\left(  \alpha\right)  =2$ since
$b_{n}$ is bounded. However we can show that, for every $s\geq2,$ it is
possible to construct a sequence $n_{k}$ such that $\mu\left(  \alpha\right)
=s.$ Indeed, we know that the expansion in RCF of $\alpha=\lim P_{n}/Q_{n}$ is
given by (\ref{Trans10}), and therefore by (\ref{Form}) we see that%
\begin{equation}
\mu\left(  \alpha\right)  =2+\limsup_{n\rightarrow\infty}\frac{\log c_{n+1}%
}{\log Q_{n}}=2+\limsup_{n\rightarrow\infty}\frac{\log c_{2k+1}}{\log Q_{2k}},
\label{mualpha}%
\end{equation}
since $c_{2k}=1.$ But $Q_{2k}$ depends only on $c_{0},c_{1},\ldots,c_{2k}$ and
therefore only on $n_{1},n_{2},\ldots,n_{k}$ by (\ref{ck}). Hence it is
possible to construct by induction the sequence $(n_{k})$ in such a way that%
\begin{equation}
c_{2k+1}=n_{k+1}-n_{k}=\left\lfloor Q_{2k}^{s-2}\right\rfloor . \label{Rec}%
\end{equation}
As $c_{2k+1}\leq Q_{2k}^{s-2}<c_{2k+1}+1,$ by (\ref{mualpha}) we get
$\mu\left(  \alpha\right)  =s.$ This contradiction proves that (\ref{Form}) is
not valid in general for NCF. As a consequence, we have

\begin{theorem}
\label{ThHancl}For any $s\in\left\{  1\right\}  \cup\left[  2,+\infty\right[
,$ there exists a NCF
\[
\alpha=-\frac{1}{b_{1}}%
\genfrac{}{}{0pt}{}{{}}{-}%
\frac{1}{b_{2}}%
\genfrac{}{}{0pt}{}{{}}{-\cdots-}%
\frac{1}{b_{n}}%
\genfrac{}{}{0pt}{}{{}}{+\cdots}%
\]
with $b_{n}\in\left\{  2,3\right\}  $ for all $n\geq1$ such that $\mu\left(
\alpha\right)  =s.$
\end{theorem}

Indeed, the case $s=1$ is given by (\ref{Omega}). This result, as well as
Theorem \ref{ThSec4} in the next section and Example 4 in Section \ref{SecEx},
has to be compared with \cite[Theorem 3.4]{Hancl}.

\section{Lehner continued fractions}

\label{SecLCF}In this section we will prove a formula similar to
(\ref{mualpha}) by transforming LCF into RCF. Let $\alpha$ defined by
(\ref{Int1}) be any LCF. Then%
\begin{align*}
\left(  a_{n},b_{n}\right)   &  =\left(  1,1\right)  \quad\Rightarrow
\quad\left(  a_{n+1},b_{n+1}\right)  =\left(  1,1\right)  \text{ or }\left(
1,2\right)  ,\\
\left(  a_{n},b_{n}\right)   &  =\left(  1,2\right)  \quad\Rightarrow
\quad\left(  a_{n+1},b_{n+1}\right)  =\left(  -1,1\right)  \text{ or }\left(
-1,2\right)  ,\\
\left(  a_{n},b_{n}\right)   &  =\left(  -1,1\right)  \quad\Rightarrow
\quad\left(  a_{n+1},b_{n+1}\right)  =\left(  1,1\right)  \text{ or }\left(
1,2\right)  ,\\
\left(  a_{n},b_{n}\right)   &  =\left(  -1,2\right)  \quad\Rightarrow
\quad\left(  a_{n+1},b_{n+1}\right)  =\left(  -1,1\right)  \text{ or }\left(
-1,2\right)  .
\end{align*}
Moreover by (\ref{Cond1}) there exist infinitely many $n$ such that $b_{n}=1.$
So there exist sequences of integers $l_{n}\geq0$ and $m_{n}\geq0$ such that
$\alpha=1+\beta$ or $\alpha=2-\beta,$ where%
\begin{multline*}
\beta:=\underset{l_{0}\text{ \textit{times}}}{\underline{\frac{1}{1}%
\genfrac{}{}{0pt}{}{{}}{+\cdots+}%
\frac{1}{1}}}%
\genfrac{}{}{0pt}{}{{}}{+}%
\frac{1}{2}%
\genfrac{}{}{0pt}{}{{}}{+}%
\underset{m_{1}\text{ \textit{times}}}{\underline{\frac{-1}{2}%
\genfrac{}{}{0pt}{}{{}}{+\cdots+}%
\frac{-1}{2}}}%
\genfrac{}{}{0pt}{}{{}}{+}%
\frac{-1}{1}%
\genfrac{}{}{0pt}{}{{}}{+}%
\underset{l_{1}\text{ \textit{times}}}{\underline{\frac{1}{1}%
\genfrac{}{}{0pt}{}{{}}{+\cdots+}%
\frac{1}{1}}}\\%
\genfrac{}{}{0pt}{}{{}}{+}%
\frac{1}{2}%
\genfrac{}{}{0pt}{}{{}}{+}%
\underset{m_{2}\text{ \textit{times}}}{\underline{\frac{-1}{2}%
\genfrac{}{}{0pt}{}{{}}{+\cdots+}%
\frac{-1}{2}}%
\genfrac{}{}{0pt}{}{{}}{+}%
}\frac{-1}{1}%
\genfrac{}{}{0pt}{}{{}}{+}%
\underset{l_{2}\text{ \textit{times}}}{\underline{\frac{1}{1}%
\genfrac{}{}{0pt}{}{{}}{+\cdots+}%
\frac{1}{1}}}%
\genfrac{}{}{0pt}{}{{}}{+}%
\frac{1}{2}%
\genfrac{}{}{0pt}{}{{}}{+}%
\underset{m_{3}\text{ \textit{times}}}{\underline{\frac{-1}{2}%
\genfrac{}{}{0pt}{}{{}}{+\cdots+}%
\frac{-1}{2}}%
\genfrac{}{}{0pt}{}{{}}{+\cdots}%
}%
\end{multline*}
It is easily seen that%
\[
\frac{1}{y}%
\genfrac{}{}{0pt}{}{{}}{+}%
\frac{-1}{1}%
\genfrac{}{}{0pt}{}{{}}{+}%
\frac{1}{x}=\frac{1}{y-1}%
\genfrac{}{}{0pt}{}{{}}{+}%
\frac{1}{1+x}.
\]
Using this repeatedly we obtain by induction for all $m\geq0$%
\[
\frac{1}{2}%
\genfrac{}{}{0pt}{}{{}}{+}%
\underset{m\text{ \textit{times}}}{\underline{\frac{-1}{2}%
\genfrac{}{}{0pt}{}{{}}{+\cdots+}%
\frac{-1}{2}}}%
\genfrac{}{}{0pt}{}{{}}{+}%
\frac{-1}{1}%
\genfrac{}{}{0pt}{}{{}}{+}%
\frac{1}{x}=\frac{1}{1}%
\genfrac{}{}{0pt}{}{{}}{+}%
\frac{1}{m+1+x},
\]
and therefore for all $l,m\geq0$%
\begin{equation}
\underset{l\text{ \textit{times}}}{\underline{\frac{1}{1}%
\genfrac{}{}{0pt}{}{{}}{+\cdots+}%
\frac{1}{1}}}%
\genfrac{}{}{0pt}{}{{}}{+}%
\frac{1}{2}%
\genfrac{}{}{0pt}{}{{}}{+}%
\underset{m\text{ \textit{times}}}{\underline{\frac{-1}{2}%
\genfrac{}{}{0pt}{}{{}}{+\cdots+}%
\frac{-1}{2}}}%
\genfrac{}{}{0pt}{}{{}}{+}%
\frac{-1}{1}%
\genfrac{}{}{0pt}{}{{}}{+}%
\frac{1}{x}=\underset{l+1\text{ \textit{times}}}{\underline{\frac{1}{1}%
\genfrac{}{}{0pt}{}{{}}{+\cdots+}%
\frac{1}{1}}}%
\genfrac{}{}{0pt}{}{{}}{+}%
\frac{1}{m+1+x}. \label{Transf2}%
\end{equation}
Let us consider the convergent $T_{n_{k}}$ of $\beta$ defined by%
\begin{multline*}
T_{n_{k}}:=\underset{l_{0}\text{ \textit{times}}}{\underline{\frac{1}{1}%
\genfrac{}{}{0pt}{}{{}}{+\cdots+}%
\frac{1}{1}}}%
\genfrac{}{}{0pt}{}{{}}{+}%
\frac{1}{2}%
\genfrac{}{}{0pt}{}{{}}{+}%
\underset{m_{1}\text{ \textit{times}}}{\underline{\frac{-1}{2}%
\genfrac{}{}{0pt}{}{{}}{+\cdots+}%
\frac{-1}{2}}}%
\genfrac{}{}{0pt}{}{{}}{+}%
\frac{-1}{1}%
\genfrac{}{}{0pt}{}{{}}{+}%
\underset{l_{1}\text{ \textit{times}}}{\underline{\frac{1}{1}%
\genfrac{}{}{0pt}{}{{}}{+\cdots+}%
\frac{1}{1}}}%
\genfrac{}{}{0pt}{}{{}}{+\cdots}%
\\%
\genfrac{}{}{0pt}{}{{}}{\cdots+}%
\frac{-1}{1}\underset{l_{k-1}\text{ \textit{times}}}{%
\genfrac{}{}{0pt}{}{{}}{+}%
\underline{\frac{1}{1}%
\genfrac{}{}{0pt}{}{{}}{+\cdots+}%
\frac{1}{1}}}%
\genfrac{}{}{0pt}{}{{}}{+}%
\frac{1}{2}%
\genfrac{}{}{0pt}{}{{}}{+}%
\underset{m_{k}\text{ \textit{times}}}{\underline{\frac{-1}{2}%
\genfrac{}{}{0pt}{}{{}}{+\cdots+}%
\frac{-1}{2}}%
\genfrac{}{}{0pt}{}{{}}{+}%
}\frac{-1}{1}.
\end{multline*}
Applying (\ref{Transf2}) with $x=\infty$ to the end of $T_{n_{k}},$ we see
that%
\begin{multline*}
T_{n_{k}}=\underset{l_{0}\text{ \textit{times}}}{\underline{\frac{1}{1}%
\genfrac{}{}{0pt}{}{{}}{+\cdots+}%
\frac{1}{1}}}%
\genfrac{}{}{0pt}{}{{}}{+}%
\frac{1}{2}%
\genfrac{}{}{0pt}{}{{}}{+}%
\underset{m_{1}\text{ \textit{times}}}{\underline{\frac{-1}{2}%
\genfrac{}{}{0pt}{}{{}}{+\cdots+}%
\frac{-1}{2}}}%
\genfrac{}{}{0pt}{}{{}}{+}%
\frac{-1}{1}%
\genfrac{}{}{0pt}{}{{}}{+}%
\underset{l_{1}\text{ \textit{times}}}{\underline{\frac{1}{1}%
\genfrac{}{}{0pt}{}{{}}{+\cdots+}%
\frac{1}{1}}}%
\genfrac{}{}{0pt}{}{{}}{+\cdots}%
\\%
\genfrac{}{}{0pt}{}{{}}{+\cdots}%
\genfrac{}{}{0pt}{}{{}}{+}%
\underset{l_{k-2}\text{ \textit{times}}}{\underline{\frac{1}{1}%
\genfrac{}{}{0pt}{}{{}}{+\cdots+}%
\frac{1}{1}}}%
\genfrac{}{}{0pt}{}{{}}{+}%
\frac{1}{2}%
\genfrac{}{}{0pt}{}{{}}{+}%
\underset{m_{k-1}\text{ \textit{times}}}{\underline{\frac{-1}{2}%
\genfrac{}{}{0pt}{}{{}}{+\cdots+}%
\frac{-1}{2}}%
\genfrac{}{}{0pt}{}{{}}{+}%
}\frac{-1}{1}%
\genfrac{}{}{0pt}{}{{}}{+}%
\underset{l_{k-1}+1\text{ \textit{times}}}{\underline{\frac{1}{1}%
\genfrac{}{}{0pt}{}{{}}{+\cdots+}%
\frac{1}{1}}}.
\end{multline*}
Applying again (\ref{Transf2}), we get
\begin{multline*}
T_{n_{k}}=\underset{l_{0}\text{ \textit{times}}}{\underline{\frac{1}{1}%
\genfrac{}{}{0pt}{}{{}}{+\cdots+}%
\frac{1}{1}}}%
\genfrac{}{}{0pt}{}{{}}{+}%
\frac{1}{1}%
\genfrac{}{}{0pt}{}{{}}{+}%
\underset{m_{1}\text{ \textit{times}}}{\underline{\frac{1}{m_{1}+1}%
\genfrac{}{}{0pt}{}{{}}{+\cdots+}%
\frac{-1}{2}}}%
\genfrac{}{}{0pt}{}{{}}{+}%
\frac{-1}{1}%
\genfrac{}{}{0pt}{}{{}}{+}%
\underset{l_{1}\text{ \textit{times}}}{\underline{\frac{1}{1}%
\genfrac{}{}{0pt}{}{{}}{+\cdots+}%
\frac{1}{1}}}%
\genfrac{}{}{0pt}{}{{}}{+\cdots}%
\\%
\genfrac{}{}{0pt}{}{{}}{+\cdots}%
\genfrac{}{}{0pt}{}{{}}{+}%
\underset{l_{k-2}+1\text{ \textit{times}}}{\underline{\frac{1}{1}%
\genfrac{}{}{0pt}{}{{}}{+\cdots+}%
\frac{1}{1}}}%
\genfrac{}{}{0pt}{}{{}}{+}%
\frac{1}{m_{k-1}+2}\underset{l_{k-1}\text{ \textit{times}}}{%
\genfrac{}{}{0pt}{}{{}}{+}%
\underline{\frac{1}{1}%
\genfrac{}{}{0pt}{}{{}}{+\cdots+}%
\frac{1}{1}}},
\end{multline*}
and so finally%
\[
T_{n_{k}}=[0,\underset{l_{0}+1\text{ \textit{times}}}{\underline{1,\ldots,1}%
},m_{1}+2,\underset{l_{1}\text{ \textit{times}}}{\underline{1,\ldots,1}}%
,m_{2}+2,,\ldots,\underset{l_{k-2}\text{ \textit{times}}}{\underline{1,\ldots
,1}},m_{k-1}+2,\underset{l_{k-1}\text{ \textit{times}}}{\underline{1,\ldots
,1}}],
\]
which yields immediately the expansion in RCF%
\[
\beta=[0,\underset{l_{0}+1\text{ \textit{times}}}{\underline{1,\ldots,1}%
},m_{1}+2,\underset{l_{1}\text{ \textit{times}}}{\underline{1,\ldots,1}}%
,m_{2}+2,\ldots,\underset{l_{k-1}\text{ \textit{times}}}{\underline{1,\ldots
,1}},m_{k}+2,\ldots].
\]
Define $t_{k}=l_{0}+l_{1}+\cdots+l_{k-1}+k.$ Applying (\ref{Form}) we get%
\begin{equation}
\mu\left(  \alpha\right)  =\mu\left(  \beta\right)  =2+\limsup_{k\rightarrow
\infty}\frac{\log\left(  m_{k}+2\right)  }{\log Q_{t_{k}}}. \label{Yes1}%
\end{equation}
For example, let $l_{k}=0$ for all $k\geq0$ and%
\begin{equation}
\alpha:=1+\frac{1}{2}%
\genfrac{}{}{0pt}{}{{}}{+}%
\underset{m_{1}\text{ \textit{times}}}{\underline{\frac{-1}{2}%
\genfrac{}{}{0pt}{}{{}}{+\cdots+}%
\frac{-1}{2}}}%
\genfrac{}{}{0pt}{}{{}}{+}%
\frac{-1}{1}%
\genfrac{}{}{0pt}{}{{}}{+}%
\frac{1}{2}%
\genfrac{}{}{0pt}{}{{}}{+}%
\underset{m_{2}\text{ \textit{times}}}{\underline{\frac{-1}{2}%
\genfrac{}{}{0pt}{}{{}}{+\cdots+}%
\frac{-1}{2}}%
\genfrac{}{}{0pt}{}{{}}{+}%
}\frac{-1}{1}%
\genfrac{}{}{0pt}{}{{}}{+}%
\frac{1}{2}%
\genfrac{}{}{0pt}{}{{}}{+\cdots}
\label{ExAlpha1}%
\end{equation}
Then $\alpha=\left[  1,1,m_{1}+2,m_{2}+2,m_{3}+2,\ldots,m_{k}+2,\ldots\right]
.$ By (\ref{Yes1}) we have%
\begin{equation}
\mu\left(  \alpha\right)  =2+\limsup_{k\rightarrow\infty}\frac{\log\left(
m_{k}+2\right)  }{\log Q_{k}}. \label{Mu1}%
\end{equation}
Let $s>2.$ As in (\ref{Rec}), for every $k$ we can choose $m_{k}$ such that%
\[
m_{k}+2=\left\lfloor Q_{k}^{s-2}\right\rfloor ,
\]
and so by (\ref{Mu1}) we get $\mu\left(  \alpha\right)  =s>2.$ We have proved

\begin{theorem}
\label{ThSec4}Let $s>2$ and let $\alpha$ be defined by (\ref{ExAlpha1}). Then
there exists a sequence $m_{k}$ such that $\mu\left(  \alpha\right)  =s.$
\end{theorem}

\section{A general formula for the irrationality exponent}

\label{SecIrrExp}In this section, we prove the following theorem, which allows
to compute the irrationality exponent of a real number $\alpha$ by knowing
suitable rational approximations.

\begin{theorem}
\label{ThIrratGen}Let $\alpha\in\mathbb{R}.$ Let $p_{n}\in\mathbb{Z}$ and
$q_{n}\in\mathbb{Z}_{>0}$ be coprime for all $n\geq0,$ with $\lim
_{n\rightarrow\infty}q_{n}=+\infty.$ Assume that there exist positive
constants $\rho,$ $\sigma$ and $\tau$ such that $q_{n+1}\geq\rho q_{n}$ and%
\begin{equation}
\frac{\sigma}{q_{n}q_{n+1}}\leq\left\vert \alpha-\frac{p_{n}}{q_{n}%
}\right\vert \leq\frac{\tau}{q_{n}q_{n+1}} \label{Daniel1}%
\end{equation}
for all large $n$. Then $\alpha$ is irrational and its irrationality exponent
is%
\begin{equation}
\mu\left(  \alpha\right)  =1+\limsup_{n\rightarrow\infty}\frac{\log q_{n+1}%
}{\log q_{n}}. \label{Daniel2}%
\end{equation}

\end{theorem}

\begin{proof}
There is no loss of generality in assuming that $q_{n+1}\geq\rho q_{n}$ and
(\ref{Daniel1}) hold for every $n\geq0.$ Let $\lambda_{n}:=\log q_{n+1}/\log
q_{n}$ and $\lambda:=\limsup_{n\rightarrow\infty}\lambda_{n}.$ Let
$\varepsilon>0.$ By (\ref{Daniel1}), we have%
\[
\left\vert \alpha-\frac{p_{n}}{q_{n}}\right\vert \leq\frac{\tau}%
{q_{n}^{\lambda_{n}+1}}\leq\frac{1}{q_{n}^{\lambda_{n}+1-\varepsilon/2}}%
\]
for all large $n$ since $q_{n}\rightarrow+\infty.$ If $\lambda=\infty,$ then
$\mu\left(  \alpha\right)  =\infty$ and (\ref{Daniel2}) holds. So we may
assume that $\lambda<\infty.$ As $\lambda_{n}\geq\lambda-\varepsilon/2$ for
infinitely many $n,$ the inequality%
\begin{equation}
\left\vert \alpha-\frac{p_{n}}{q_{n}}\right\vert \leq\frac{1}{q_{n}%
^{\lambda+1-\varepsilon}} \label{Daniel2.1}%
\end{equation}
holds for infinitely many $n.$ Hence the inequality%
\begin{equation}
\left\vert \alpha-\frac{p}{q}\right\vert \leq\frac{1}{q^{\lambda
+1-\varepsilon}} \label{Daniel2.2}%
\end{equation}
has infinitely many solutions $p/q$ because $p_{n}$ and $q_{n}$ are coprime
and $q_{n}\rightarrow+\infty.$ This proves that $\mu\left(  \alpha\right)
\geq\lambda+1.$ We prove that $\lambda+1\geq\mu\left(  \alpha\right)  .$ Let
again $\varepsilon>0.$ We consider the inequality%
\begin{equation}
\left\vert \alpha-\frac{p}{q}\right\vert \leq\frac{1}{q^{\lambda
+1+\varepsilon}}. \label{Daniel3}%
\end{equation}
Our aim is to prove that this inequality has only finitely many solutions
$p/q.$ Without loss of generality, we can assume that $p$ and $q>0$ are
coprime. We distinguish two cases.

\textit{First case} : there exists $n$ such that $p/q=p_{n}/q_{n}.$ Then
$p_{n}=p$ and $q_{n}=q$ since $p/q$ and $p_{n}/q_{n}$ are both irreducible. By
(\ref{Daniel1}), we have%
\[
\left\vert \alpha-\frac{p}{q}\right\vert =\left\vert \alpha-\frac{p_{n}}%
{q_{n}}\right\vert \geq\frac{\sigma}{q_{n}^{\lambda_{n}+1}}=\frac{\sigma
}{q^{\lambda_{n}+1}}=\frac{\sigma q^{\varepsilon/2}}{q^{\lambda_{n}%
+1+\varepsilon/2}}.
\]
For large $q,$ we have $\sigma q^{\varepsilon/2}\geq1$ and $\lambda
_{n}+1+\varepsilon/2\leq\lambda+1+\varepsilon$ since $\lim_{n\rightarrow
\infty}q_{n}=+\infty.$ Hence
\begin{equation}
\left\vert \alpha-\frac{p}{q}\right\vert \geq\frac{1}{q^{\lambda
+1+\varepsilon}}. \label{Daniel5}%
\end{equation}

\textit{Second case} : for all $n,$ we have $pq_{n}-qp_{n}\neq0.$ Fix a
positive integer $Q$ such that $2\tau Q>\rho q_{1}$ and assume that $q\geq Q.$
Let $n=n(q)\geq1$ be the least integer such that%
\begin{equation}
\frac{\tau}{\rho q_{k}}q<\frac{1}{2} \label{5.13}%
\end{equation}
for all $k\geq n+1,$ which exists since $\lim_{k\rightarrow\infty}%
q_{k}=+\infty.$ Then%
\begin{equation}
q_{n}\leq\frac{2\tau}{\rho}q. \label{Majq_n}%
\end{equation}
As $pq_{n+1}-qp_{n+1}\neq0,$ we can write by (\ref{Daniel1}) and (\ref{5.13}):%
\begin{align*}
1  &  \leq\left\vert pq_{n+1}-qp_{n+1}\right\vert \leq q\left\vert
q_{n+1}\alpha-p_{n+1}\right\vert +q_{n+1}\left\vert q\alpha-p\right\vert \\
&  \leq\frac{\tau}{q_{n+2}}q+q_{n+1}\left\vert q\alpha-p\right\vert \leq
\frac{\tau}{\rho q_{n+1}}q+q_{n+1}\left\vert q\alpha-p\right\vert <\frac{1}%
{2}+qq_{n+1}\left\vert \alpha-\frac{p}{q}\right\vert .
\end{align*}
By using (\ref{Majq_n}), this implies%
\[
\left\vert \alpha-\frac{p}{q}\right\vert >\frac{1}{2qq_{n+1}}=\frac{1}%
{2qq_{n}^{\lambda_{n}}}\geq\frac{\rho^{\lambda_{n}}}{2^{\lambda_{n}+1}%
\tau^{\lambda_{n}}q^{\lambda_{n}+1}}=\frac{\rho^{\lambda_{n}}q^{\varepsilon
/2}}{2^{\lambda_{n}+1}\tau^{\lambda_{n}}q^{\lambda_{n}+1+\varepsilon/2}}.
\]
However there exist two constants $C$ and $C^{\prime}$ such that $C\leq
\lambda_{n}\leq C^{\prime},$ because $\lambda$ is finite, $q_{n+1}\geq\rho
q_{n}$ and $q_{n}\rightarrow+\infty.$ So there exists $C^{\prime\prime}>0$
such that%
\[
\left\vert \alpha-\frac{p}{q}\right\vert >\frac{C^{\prime\prime}%
q^{\varepsilon/2}}{q^{\lambda_{n}+1+\varepsilon/2}}.
\]
Now $n=n\left(  q\right)  $ tends to infinity with $q,$ and therefore for
$q\geq Q^{\prime}>Q$ we have $C^{\prime\prime}q^{\varepsilon/2}\geq1$ and
$\lambda_{n}\leq\lambda+\varepsilon/2.$ This yields for large $q$%
\begin{equation}
\left\vert \alpha-\frac{p}{q}\right\vert >\frac{1}{q^{\lambda+1+\varepsilon}}.
\label{Daniel6}%
\end{equation}
By (\ref{Daniel5}) and (\ref{Daniel6}) we conclude that the inequality
(\ref{Daniel3}) has only finitely many solutions for every $\varepsilon>0.$
The proof of Theorem \ref{ThIrratGen} is complete.
\end{proof}

\section{Proof of Theorem \ref{Th1} and Corollary \ref{Cor1.1}}

\label{SecProof1}First we prepare two lemmas.

\begin{lemma}
\label{LemNegative}Let $m$ and $h$ be positive integers. Assume that $a_{m}=1
$ and $a_{m+k}=-1$ for $1\leq k\leq h.$ Then%
\begin{align}
q_{m+k}  &  \geq\frac{k+2}{k+1}q_{m+k-1}\qquad\left(  0\leq k\leq h-1\right)
,\label{Help3.2}\\
q_{m+h}  &  \geq\frac{b_{m+h}}{h+1}q_{m+h-1}. \label{Help4.3}%
\end{align}

\end{lemma}

\begin{proof}
First we prove (\ref{Help3.2}). It is true for $h=1.$ Indeed, in this case we
have $a_{m}=1,$ $a_{m+1}=-1$ and $b_{m}\geq2,$ whence $q_{m}=b_{m}%
q_{m-1}+q_{m-2}\geq2q_{m-1}.$ Assume that (\ref{Help3.2}) is satisfied at the
order $h\geq1,$ and prove it at the order $h+1.$ By (\ref{Help3.2}),
\[
q_{m+h-2}\leq\frac{h}{h+1}q_{m+h-1}.
\]
Since $a_{m+h}=a_{m+h+1}=-1,$ we have $b_{m+h}\geq2$ and%
\[
q_{m+h}\geq2q_{m+h-1}-\frac{h}{h+1}q_{m+h-1}=\frac{h+2}{h+1}q_{m+h-1},
\]
which proves (\ref{Help3.2}) by induction. By using (\ref{Help3.2}) we get%
\[
q_{m+h}\geq b_{m+h}\left(  q_{m+h-1}-q_{m+h-2}\right)  \geq b_{m+h}%
q_{m+h-1}\left(  1-\frac{h}{h+1}\right)  ,
\]
which proves (\ref{Help4.3}).
\end{proof}

\begin{lemma}
\label{Lemb_n}Let $m$ and $h$ be positive integers such that $b_{m}\geq3$ and
$b_{m+k}=2$ for $1\leq k\leq h$. Assume that $b_{m-1}\geq2$ if $m\geq2.$ Then%
\begin{equation}
q_{m+k}\geq\frac{k+2}{k+1}q_{m+k-1}\qquad\left(  0\leq k\leq h\right)  .
\label{Rec_A}%
\end{equation}

\end{lemma}

\begin{proof}
By Proposition \ref{Lem1}, we see that $q_{m-2}<q_{m-1}$ for $m\geq2,$ and
this remains true for $m=1.$ Therefore for $m\geq1$ we have
\[
q_{m}\geq b_{m}q_{m-1}-q_{m-2}\geq(b_{m}-1)q_{m-1}\geq2q_{m-1}.
\]
Hence (\ref{Rec_A}) is true for $k=0.$ Assume that (\ref{Rec_A}) is true for
some $k\leq h-1.$ Then%
\[
q_{m+k+1}\geq2q_{m+k}-q_{m+k-1}\geq\left(  2-\frac{k+1}{k+2}\right)
q_{m+k}=\frac{k+3}{k+2}q_{m+k},
\]
which proves (\ref{Rec_A}) by induction.
\end{proof}

\textbf{Proof of Theorem \ref{Th1}.} We observe that%
\begin{align*}
\limsup_{n\rightarrow\infty}\frac{\log q_{n+1}}{\log q_{n}}  &  =1+\limsup
_{n\rightarrow\infty}\frac{\log b_{n+1}+\log q_{n}+\log\left(  1+a_{n+1}%
b_{n+1}^{-1}q_{n-1}q_{n}^{-1}\right)  }{\log q_{n}}\\
&  =2+\limsup_{n\rightarrow\infty}\frac{\log b_{n+1}}{\log q_{n}}%
\end{align*}
if $q_{n-1}q_{n}^{-1}$ is bounded. Hence we only have to check that Theorem
\ref{ThIrratGen} applies when $(A),$ $(B),$ $(C)$ or $(D)$ are satisfied.
Besides, it results from (\ref{Ksi2}) that (\ref{Daniel1}) holds if there
exist two positive constants $\gamma$ and $\delta$ such that $\gamma
q_{n+1}\leq\xi_{n+1}\leq\delta q_{n+1}$ for all large $n.\smallskip$

Assume that $(A)$ is satisfied. Let $n$ large. As the number of consecutive
$k$ such that $a_{k}=-1$ is bounded, there exists $p\leq n$ such that
$a_{p}=1.$ We distinguish four cases.

\textit{Case 1.} $a_{n+1}=a_{n+2}=1.$ Then $0<x_{n+1}\leq1$ by Proposition
\ref{Lem-x-n,k}, so that%
\begin{equation}
q_{n+1}\leq\xi_{n+1}\leq q_{n+1}+q_{n}. \label{Approx1.1}%
\end{equation}
As $q_{n+1}\geq q_{n}$ by Proposition \ref{Lem1}, we obtain $q_{n+1}\leq
\xi_{n+1}\leq2q_{n+1}.$

\textit{Case 2.} $a_{n+1}=-a_{n+2}=1.$ Then $-1\leq x_{n+1}<0$ by Proposition
\ref{Lem-x-n,k}, whence%
\begin{equation}
q_{n+1}-q_{n}\leq\xi_{n+1}\leq q_{n+1}. \label{Approx2}%
\end{equation}
By (\ref{Help3.2}) with $m=n+1,$ $h=1$ and $k=0,$ we see that $q_{n+1}%
\geq2q_{n},$ and so%
\[
\frac{1}{2}q_{n+1}\leq\xi_{n+1}\leq q_{n+1}.
\]

\textit{Case 3.} $a_{n+1}=-a_{n+2}=-1.$ Then (\ref{Approx1.1}) holds. Let
$m\leq n$ be the greatest integer such that $a_{m}=1.$ By (\ref{Help4.3}) with
$h=n+1-m,$ we have%
\[
q_{n+1}\geq\frac{1}{h+1}q_{n}\geq\frac{1}{L}q_{n}%
\]
as $h<L$ for some $L\geq2$ since the number of consecutive $k$ such that
$a_{m+k}=-1$ is bounded. So we obtain $q_{n+1}\leq\xi_{n+1}\leq\left(
L+1\right)  q_{n+1}.$

\textit{Case 4.} $a_{n+1}=a_{n+2}=-1.$ Then (\ref{Approx2}) holds. Let $m\leq
n$ be the greatest integer such that $a_{m}=1.$ By (\ref{Help3.2}) with
$h=n+2-m$ and $k=h-1,$ we see that%
\[
q_{n+1}\geq\left(  1+\frac{1}{h}\right)  q_{n}\geq\frac{L+1}{L}q_{n}.
\]
By (\ref{Approx2}), we obtain%
\[
\frac{1}{L+1}q_{n+1}\leq\xi_{n+1}\leq q_{n+1}.
\]
Theorem \ref{ThIrratGen} applies with $\rho=1/L,$ $\sigma=1/(L+1)$ and
$\tau=L+1.\smallskip$

Assume that $(B)$ is satisfied. Let $n$ large. As $b_{n+1}\geq2,$ we have
$q_{n+1}>q_{n}$ by Proposition \ref{Lem1}. Therefore by Proposition
\ref{Lem-x-n,k}
\begin{equation}
\xi_{n+1}\leq q_{n+1}+q_{n}\leq2q_{n+1}. \label{AA}%
\end{equation}
For the lower bound of $\xi_{n+1},$ we distinguish two cases.

\textit{Case 1.} If $b_{n+1}\geq3,$ then $q_{n+1}\geq b_{n+1}q_{n}-q_{n-1}%
\geq2q_{n}.$

\textit{Case 2. }If $b_{n+1}=2,$ we define $m\leq n$ to be the greatest
integer such that $b_{m}\geq3.$ By (\ref{Rec_A}) with $k=h=n+1-m,$ we have%
\begin{equation}
q_{n+1}\geq\frac{h+2}{h+1}q_{n}\geq\frac{M+1}{M}q_{n} \label{MajAh}%
\end{equation}
as $h<M$ for some $M\geq2$ since the number of consecutive $k$ such that
$b_{m+k}=2$ is bounded.

So in both cases (\ref{MajAh}) holds and
\begin{equation}
\xi_{n+1}\geq q_{n+1}-q_{n}\geq q_{n+1}-\frac{M}{M+1}q_{n+1}\geq\frac{1}%
{M+1}q_{n+1}. \label{BB}%
\end{equation}
Theorem \ref{ThIrratGen} applies with $\rho=1,$ $\sigma=1/2$ and
$\tau=M+1.\smallskip$

Assume that $(C)$ is satisfied. Then $q_{n+1}>q_{n}$ for $n$ large by
Proposition \ref{Lem1}. Therefore (\ref{AA}) holds. For the lower bound of
$\xi_{n+1},$ we distinguish three cases.

\textit{Case 1}. If $a_{n+2}=1,$ then $x_{n+1}\geq0$ by Proposition
\ref{Lem-x-n,k}, whence $\xi_{n+1}\geq q_{n+1}.$

\textit{Case 2}. If $a_{n+2}=-a_{n+1}=-1,$ then $b_{n+1}\geq2$ and
$q_{n+1}\geq2q_{n}.$ Therefore%
\begin{equation}
\xi_{n+1}\geq q_{n+1}-q_{n}\geq\frac{1}{2}q_{n+1}. \label{Min2}%
\end{equation}

\textit{Case 3}. If $a_{n+2}=$ $a_{n+1}=-1,$ then (\ref{Min2}) holds again,
since
\[
q_{n+1}=\left(  b_{n+1}+a_{n+1}\right)  q_{n}-a_{n+1}\left(  q_{n}%
-q_{n-1}\right)  \geq2q_{n}.
\]

Theorem \ref{ThIrratGen} applies with $\rho=1,$ $\sigma=1/2$ and
$\tau=2.\smallskip\smallskip$

Assume that $(D)$ is satisfied. Again, $q_{n+1}>q_{n}$ for $n$ large by
Proposition \ref{Lem1} and (\ref{AA}) holds again. For the lower bound of
$\xi_{n+1},$ we distinguish two cases.

\textit{Case 1}. If $a_{n+2}=1,$ then $\xi_{n+1}\geq q_{n+1}.$

\textit{Case 2}. If $a_{n+2}=-1$, then $b_{n+1}\geq3$ and $q_{n+1}\geq
b_{n+1}q_{n}-q_{n-1}\geq2q_{n}$, so that (\ref{Min2}) holds.

Theorem \ref{ThIrratGen} applies with $\rho=1,$ $\sigma=1/2$ and
$\tau=2.\smallskip\medskip$

\textbf{Proof of Corollary \ref{Cor1.1}}

Assume that $(A),$ $(B)$, $(C)$ or $(D)$ is satisfied. We will prove that
there exists a positive constant $K$ such that%
\begin{equation}
K^{n}\prod_{k=1}^{n}b_{k}\leq q_{n}\leq F_{n+1}\prod_{k=1}^{n}b_{k}%
\quad\left(  n\geq0\right)  , \label{Encad}%
\end{equation}
where $F_{n}$ is the \textit{n}-th Fibonacci number. It is clear that
Corollary \ref{Cor1.1} results immediately from Theorem \ref{Th1} and
(\ref{Encad}). The second inequality in (\ref{Encad}) comes from the fact that
$q_{0}=1,$ $q_{1}=b_{1}$ and%
\[
q_{n+1}=b_{n+1}q_{n}+a_{n}q_{n-1}\leq q_{n}+q_{n-1}\quad\left(  n\geq1\right)
.
\]
Now we prove the first inequality in (\ref{Encad}) in each of the cases $(A),$
$(B)$, $(C)$ or $(D).$ For this, it is sufficient to prove the existence of a
positive number $\kappa$ such that $q_{n+1}\geq\kappa b_{n+1}q_{n}$ for all
large $n.\smallskip$

Assume that $(A)$ is satisfied, which means that the number of consecutive $n$
such that $a_{n}=-1$ is less than $L$ $\left(  \geq2\right)  .$

If $a_{n+1}=1,$ then $q_{n+1}\geq b_{n+1}q_{n},$ and so
\begin{equation}
q_{n+1}\geq\frac{b_{n+1}}{L}q_{n}. \label{Min01}%
\end{equation}
If $a_{n+1}=-1,$ let $m<n$ be the greatest integer such that $a_{m}=1.$ Then
(\ref{Help4.3}) with $h=n-m+1$ shows that (\ref{Min01}) also holds in this
case.\smallskip

Assume that $(B)$ is satisfied, which means that the number of consecutive $n$
such that $b_{n}=2$ is less than $M$ $\left(  \geq2\right)  .$ Then by
(\ref{MajAh}) we have%
\[
q_{n+1}\geq b_{n+1}\left(  q_{n}-q_{n-1}\right)  \geq b_{n+1}\left(
q_{n}-\frac{M}{M+1}q_{n}\right)  =\frac{b_{n+1}}{M+1}q_{n}\qquad\left(
n\geq1\right)  .
\]

Assume that $(C)$ is satisfied. If $a_{n+1}=1,$ then $q_{n+1}\geq b_{n+1}%
q_{n}.$

If $a_{n+1}=-1,$ then $b_{n+1}\geq3$. Since $q_{n}>q_{n-1},$ we get%
\[
q_{n+1}\geq\frac{1}{2}b_{n+1}q_{n}+\frac{3}{2}q_{n}-q_{n-1}\geq\frac{1}%
{2}b_{n+1}q_{n}.
\]

Assume that $(D)$ is satisfied. If $a_{n+1}=1,$ then again $q_{n+1}\geq
b_{n+1}q_{n}.$

If $a_{n+1}=-1,$ then $b_{n+1}\geq2$ and $q_{n}>q_{n-1},$ and so%
\[
q_{n+1}\geq\frac{1}{2}b_{n+1}q_{n}+q_{n}-q_{n-1}\geq\frac{1}{2}b_{n+1}q_{n}.
\]

\section{Applications}

\label{SecEx}We give here examples of semi-regular continued fractions and
compute their irrationality exponent by using Corollary \ref{Cor1.1}.\medskip

\textbf{Example 1.} For $c\neq0,-1,-2,\ldots,$ the modified Bessel function is
defined by%
\[
_{0}F_{1}\left(  c;z\right)  :=\sum_{n=0}^{\infty}\frac{z^{n}}{\left(
c\right)  _{n}n!},
\]
where $(c)_{0}=1$ and $(c)_{n}=c\left(  c+1\right)  \cdots\left(
c+n-1\right)  $ for $n\geq1.$ It is known that%
\[
\frac{_{0}F_{1}\left(  1;z\right)  }{_{0}F_{1}\left(  2;z\right)  }=1+\frac
{z}{2}%
\genfrac{}{}{0pt}{}{{}}{+}%
\frac{z}{3}%
\genfrac{}{}{0pt}{}{{}}{+}%
\frac{z}{4}%
\genfrac{}{}{0pt}{}{{}}{+}%
\frac{z}{5}%
\genfrac{}{}{0pt}{}{{}}{+\cdots}%
\]
(cf. \cite[Formula (6.1.51)]{JoTh}). Corollary \ref{Cor1.1} $(B),$ $(C)$ or
$(D)$ applies and we get%
\[
\mu\left(  1+\frac{a_{1}}{2}%
\genfrac{}{}{0pt}{}{{}}{+}%
\frac{a_{2}}{3}%
\genfrac{}{}{0pt}{}{{}}{+}%
\frac{a_{3}}{4}%
\genfrac{}{}{0pt}{}{{}}{+}%
\frac{a_{4}}{5}%
\genfrac{}{}{0pt}{}{{}}{+}%
\frac{a_{5}}{6}%
\genfrac{}{}{0pt}{}{{}}{+}%
\frac{a_{6}}{7}%
\genfrac{}{}{0pt}{}{{}}{+\cdots}%
\right)  =2
\]
for any sequence $(a_{n})_{n\geq1}$ with $a_{n}=\pm1.\medskip$

\textbf{Example 2.} For any integer $b\geq1,$ we have%
\[
e^{\frac{1}{b}}=1+\frac{1}{b}%
\genfrac{}{}{0pt}{}{{}}{-}%
\frac{1}{2}%
\genfrac{}{}{0pt}{}{{}}{+}%
\frac{1}{3b}%
\genfrac{}{}{0pt}{}{{}}{-}%
\frac{1}{2}%
\genfrac{}{}{0pt}{}{{}}{+}%
\frac{1}{5b}%
\genfrac{}{}{0pt}{}{{}}{-}%
\frac{1}{2}%
\genfrac{}{}{0pt}{}{{}}{+}%
\frac{1}{7b}%
\genfrac{}{}{0pt}{}{{}}{-}%
\frac{1}{2}%
\genfrac{}{}{0pt}{}{{}}{+\cdots}%
\]
(cf \cite[(A.61)]{Bor}). Then Corollary \ref{Cor1.1} $(A)$ yields $\mu\left(
e^{1/b}\right)  =2.$ More generally, Corollary \ref{Cor1.1} $(B)$ applies and
we get%
\[
\mu\left(  1+\frac{a_{1}}{b}%
\genfrac{}{}{0pt}{}{{}}{+}%
\frac{a_{2}}{2}%
\genfrac{}{}{0pt}{}{{}}{+}%
\frac{a_{3}}{3b}%
\genfrac{}{}{0pt}{}{{}}{+}%
\frac{a_{4}}{2}%
\genfrac{}{}{0pt}{}{{}}{+}%
\frac{a_{5}}{5b}%
\genfrac{}{}{0pt}{}{{}}{+}%
\frac{a_{6}}{2}%
\genfrac{}{}{0pt}{}{{}}{+\cdots}%
\right)  =2
\]
for any sequence $(a_{n})_{n\geq1}$ with $a_{n}=\pm1.\medskip$

\textbf{Example 3.} Let $\alpha>0$ be any irrational number with the regular
continued fraction expansion $\alpha^{-1}=\left[  c_{0};c_{1},c_{2}%
,c_{3},\ldots\right]  $ and let $b$ be any integer with $\left\vert
b\right\vert >1$. Let $p_{n}/q_{n}$ be the n-th convergent of the RCF
expansion of $\alpha^{-1}$. Adams and Davison (\cite{Adams}, see also
\cite[Th. 7.18]{Bor}) proved that%
\begin{equation}
S_{b}(\alpha):=\left(  b-1\right)  \sum_{k=1}^{\infty}\frac{1}{b^{\left\lfloor
k\alpha\right\rfloor }}=b_{0}+\frac{1}{b_{1}}%
\genfrac{}{}{0pt}{}{{}}{+}%
\frac{1}{b_{2}}%
\genfrac{}{}{0pt}{}{{}}{+}%
\frac{1}{b_{3}}%
\genfrac{}{}{0pt}{}{{}}{+\cdots}%
, \label{SRCF}%
\end{equation}
where $b_{0}=c_{0}b$ and $b_{n}=\left(  b^{q_{n}}-b^{q_{n-2}}\right)  /\left(
b^{q_{n-1}}-1\right)  $ for $n\geq1$. Note that $b_{n}$ is an integer
(positive or negative), since $q_{n-1}$ divides $q_{n}-q_{n-2}.$ It is clear
that (\ref{SRCF}) can be written as a SRCF. In particular \cite{Dav}, if
$\alpha=(1+\sqrt{5})/2,$ then $q_{n}=F_{n+1}$ for $n\geq-1,$ where $F_{n}$ is
the Fibonacci sequence defined by $F_{0}=0,$ $F_{1}=1$ and $F_{n+2}%
=F_{n+1}+F_{n}$ for all $n\geq0.$ Applying (\ref{SRCF}), we get%
\begin{align*}
S_{2}(\alpha)  &  =\frac{1}{1}%
\genfrac{}{}{0pt}{}{{}}{+}%
\frac{1}{2^{F_{1}}}%
\genfrac{}{}{0pt}{}{{}}{+}%
\frac{1}{2^{F_{2}}}%
\genfrac{}{}{0pt}{}{{}}{+}%
\frac{1}{2^{F_{3}}}%
\genfrac{}{}{0pt}{}{{}}{+}%
\frac{1}{2^{F_{4}}}%
\genfrac{}{}{0pt}{}{{}}{+}%
\frac{1}{2^{F_{5}}}%
\genfrac{}{}{0pt}{}{{}}{+\cdots}%
,\\
S_{-2}(\alpha)  &  =\frac{1}{1}%
\genfrac{}{}{0pt}{}{{}}{+}%
\frac{\left(  -1\right)  ^{F_{2}}}{2^{F_{1}}}%
\genfrac{}{}{0pt}{}{{}}{+}%
\frac{\left(  -1\right)  ^{F_{3}}}{2^{F_{2}}}%
\genfrac{}{}{0pt}{}{{}}{+}%
\frac{\left(  -1\right)  ^{F_{4}}}{2^{F_{3}}}%
\genfrac{}{}{0pt}{}{{}}{+}%
\frac{\left(  -1\right)  ^{F_{5}}}{2^{F_{4}}}%
\genfrac{}{}{0pt}{}{{}}{+\cdots}%
.
\end{align*}
Returning to the general case, it is easily seen that%
\[
\limsup_{n\rightarrow\infty}\frac{\log\left\vert b_{n+1}\right\vert }%
{\log\left\vert b_{1}b_{2}\cdots b_{n}\right\vert }=\limsup_{n\rightarrow
\infty}\left(  \frac{q_{n+1}}{q_{n}}\right)  -1.
\]
Corollary \ref{Cor1.1} $(B)$, $(C)$ or $(D)$ applies and we get%
\begin{equation}
\mu\left(  b_{0}+\frac{a_{1}}{b_{1}}%
\genfrac{}{}{0pt}{}{{}}{+}%
\frac{a_{2}}{b_{2}}%
\genfrac{}{}{0pt}{}{{}}{+}%
\frac{a_{3}}{b_{3}}%
\genfrac{}{}{0pt}{}{{}}{+\cdots}%
\right)  =1+\limsup_{n\rightarrow\infty}\left(  \frac{q_{n+1}}{q_{n}}\right)
\label{Formule}%
\end{equation}
for any sequence $(a_{n})_{n\geq1}$ with $a_{n}=\pm1.$ For $\alpha=(1+\sqrt
{5})/2,$ we have%
\begin{equation}
\mu\left(  S_{2}(\alpha)\right)  =\mu\left(  S_{-2}(\alpha)\right)
=\frac{3+\sqrt{5}}{2}=2.61803\ldots. \label{NbreOr}%
\end{equation}
Returning to the general case (\ref{Formule}), it is clear that $S_{b}%
(\alpha)$ is a Liouville number if and only if $c_{n}$ is unbounded since
$c_{n+1}<q_{n+1}/q_{n}<c_{n+1}+1$ for all $n\geq0$ by (\ref{Int5}). For
example, $S_{b}(\alpha)$ is a Liouville number if $\mu\left(  \alpha\right)
>2.$ Similarly $S_{b}(e)$ is a Liouville number.

Moreover, let $L=\limsup_{n\rightarrow\infty}\left(  q_{n+1}/q_{n}\right)  .$
Then $L^{2}\geq L+1$ by (\ref{Int5}). Therefore $L\geq\left(  1+\sqrt
{5}\right)  /2,$ the golden number. Consequently $\mu\left(  S_{b}%
(\alpha)\right)  \geq\left(  3+\sqrt{5}\right)  /2.$

Finally, assume that $\alpha$ is reduced quadratic. Then its RCF expansion is
purely periodic, so that $c_{0}=0$ and $c_{n+H}=c_{n}$ for some $H\geq1$ and
all $n\geq1.$ Hence
\[
\frac{q_{kH}}{q_{kH-1}}=\left[  c_{kH},c_{kH-1},\ldots,c_{1}\right]  =\left[
\underline{c_{H},c_{H-1},\ldots,c_{1}},\ldots,\underline{c_{H},c_{H-1}%
,\ldots,c_{1}}\right]
\]
for all $k\geq1.$ Let $\alpha^{\ast}$ and $N(\alpha)$ be the conjugate and the
norm of $\alpha,$ respectively. By applying Galois Theorem (\cite[\S \ 23]%
{Per}, \cite[Exercise 4.9]{Duv}), we get%
\begin{equation}
\lim_{k\rightarrow\infty}\frac{q_{kH}}{q_{kH-1}}=-\frac{1}{\alpha^{\ast}%
}=-\frac{\alpha}{N(\alpha)}:=\beta_{1}. \label{Formule1}%
\end{equation}
Let $\beta_{n}$ be defined by $\beta_{n}:=c_{n-1}+1/\beta_{n-1}$ for
$n=2,\ldots,H.$ Then by (\ref{Formule}) and (\ref{Formule1}) we obtain%
\[
\mu\left(  b_{0}+\frac{a_{1}}{b_{1}}%
\genfrac{}{}{0pt}{}{{}}{+}%
\frac{a_{2}}{b_{2}}%
\genfrac{}{}{0pt}{}{{}}{+}%
\frac{a_{3}}{b_{3}}%
\genfrac{}{}{0pt}{}{{}}{+\cdots}%
\right)  =1+\max\left(  \beta_{1},\ldots,\beta_{H}\right)
\]
for any sequence $(a_{n})_{n\geq1}$ with $a_{n}=\pm1.$ For example, let
$c,d\in\mathbb{Z}_{>0}$ and%
\[
\alpha=\left[  \underline{c,d},\underline{c,d},\underline{c,d},\underline{c,d}%
,\ldots\right]  .
\]
Then $d\alpha^{2}-cd\alpha-c=0$ and $\mu\left(  S_{b}(\alpha)\right)
=1+\max\left(  \alpha,\alpha d/c\right)  .$ This generalizes (\ref{NbreOr}%
).\medskip

In Example 3, we could also have used \cite[Corollary 4]{DS}, and the same is
true for Example 1 (but not for Example 2). We give a last example, where the
irrationality exponent is greater than 2 and \cite[Corollary 4]{DS} is not
applicable.\medskip

\textbf{Example 4.} Let $\sigma>1$. For $k\geq1,$ define $b_{3k-2}%
=b_{3k-1}=2,$ $b_{3k}=\left\lfloor 2^{\sigma^{k}}\right\rfloor .$ Then
\[
\limsup_{n\rightarrow\infty}\frac{\log b_{n+1}}{\log\left(  b_{1}b_{2}\cdots
b_{n}\right)  }=\limsup_{k\rightarrow\infty}\frac{\log b_{3k}}{\log\left(
b_{1}b_{2}\cdots b_{3k-1}\right)  }=\sigma-1.
\]
Corollary \ref{Cor1.1} $(B)$ applies, and we get%
\[
\mu\left(  b_{0}+\frac{a_{1}}{b_{1}}%
\genfrac{}{}{0pt}{}{{}}{+}%
\frac{a_{2}}{b_{2}}%
\genfrac{}{}{0pt}{}{{}}{+}%
\frac{a_{3}}{b_{3}}%
\genfrac{}{}{0pt}{}{{}}{+\cdots}%
\right)  =\sigma+1>2
\]
for any sequence $(a_{n})_{n\geq1}$ with $a_{n}=\pm1.$

\end{document}